\documentclass[12pt]{amsart}
\usepackage{latexsym, amsmath, amssymb, mathrsfs}
\usepackage{hyperref} %for pdflatex generating pdf file
\theoremstyle{plain}
  \newtheorem{Thm}{Theorem}[section] 
  \newtheorem{Lma}[Thm]{Lemma} 
  \newtheorem{Cor}[Thm]{Corollary} 
  \newtheorem{Prop}[Thm]{Proposition}

\theoremstyle{definition}
  \newtheorem{Def}[Thm]{Definition}
  \newtheorem{Q}[Thm]{Question}
  \newtheorem{Exa}[Thm]{Example}

  \newtheorem{Dis}[Thm]{Discussion}

\theoremstyle{remark}
  \newtheorem{Rem}[Thm]{Remark}

\newcommand{\mlabel}[1]%
  {\mbox{}\marginpar{\raggedleft\hspace{0pt}{\rm\ttfamily#1}}\label{#1}}

\newcommand{\rank}{\operatorname{rank}}

\newcommand{\ord}{\operatorname{ord}}

\newcommand{\Hom}[3]{\operatorname{Hom}_{#1}(#2,#3){}}
\newcommand{\R}[1]{{R^{({#1})}}}
\newcommand{\M}[1]{{M^{({#1})}}}
\newcommand{\A}[1]{{A^{({#1})}}}
\newcommand{\B}[1]{{B^{({#1})}}}

\newcommand{\trunc}{\vert}

\newcommand{\fm}{{\mathfrak m}}

\newcommand{\ringR}{(R,\fm,k)} %{\text{$(R,\fm,k)$}}
 \newcommand{\ann}{{\rm Ann}}

\newcommand{\Dim}{{\rm dim}}

\newcommand{\E}{E_R(k)}

\newcommand{\cx}{{\rm cx}}
\renewcommand{\hom}{\operatorname{Hom}}

\newcommand{\mI}{\mathbf I}
\newcommand{\sR}{\mathscr{R}}
\newcommand{\sF}{\mathscr{F}}
\newcommand{\sS}{\mathscr{S}}

\newcommand{\sy}[1]{^{(#1)}}

\newcommand{\segre}{\,\sharp\,} 
\newcommand{\ol}{\overline} 
\newcommand{\ul}{\underline}

\providecommand{\rank}{\operatorname{rank}}
\providecommand{\deg}{\operatorname{deg}}

\renewcommand{\ge}{\geqslant} \renewcommand{\le}{\leqslant} 
\renewcommand{\geq}{\geqslant} \renewcommand{\leq}{\leqslant} 
\renewcommand{\simeq}{\cong}
\newcounter{hours}\newcounter{minutes}

\newcommand{\excise}[1]{}
\begin{document}

\title[The Frobenius complexity of a local ring]
{\bf The Frobenius complexity of a local ring of prime characteristic}
\author[Florian~Enescu]{Florian Enescu}
\author[Yongwei~Yao]{Yongwei Yao}
\address{Department of Mathematics and Statistics, Georgia State
  University, Atlanta, GA 30303 USA} 
\email{fenescu@gsu.edu}
\email{yyao@gsu.edu}
%\thanks{2010 {\em Mathematics Subject Classification\/}: 13A35}
\thanks{The first author was partially supported by the NSA grant H98230-12-1-0206.}  

\date{}

%\begin{document}
\begin{abstract}
We introduce a new invariant for local rings of prime characteristic, called Frobenius complexity, that measures the abundance of Frobenius actions on the injective hull of the residue field of a local ring. We present an important case where the Frobenius complexity is finite, and prove that complete, normal rings of dimension two or less have Frobenius complexity less than or equal to zero. Moreover, we compute the Frobenius complexity for the determinantal ring obtained by modding out the $2 \times 2$ minors of a $2 \times 3$ matrix of indeterminates, showing that this number can be positive, irrational and depends upon the characteristic. We also settle a conjecture of Katzman, Schwede, Singh and Zhang on the infinite generation of the ring of Frobenius operators of a local normal complete $\mathbb{Q}$-Gorenstein ring.

\end{abstract}

\maketitle

\section{Introduction}

Throughout this paper $R$ is a commutative Noetherian ring, often local, of positive characteristic $p$, where $p$ is prime. Let $q=p^e$, where $e \in \mathbb{N}  = \{ 0, 1, \ldots \}$. Consider the $e$th Frobenius homomorphism $F^e:R\to R$ defined $F(r)=r^q$, for all $r \in R$. For an $R$-module $M$, an $e$th Frobenius action (or Frobenius operator) on $M$ is an additive map $\phi: M \to M$ such that $\phi(rm)= r^{p^e}\phi(m)$, for all $r \in R, m \in M$.

In recent years, there has been an interest in the study of the
Frobenius actions on the local cohomology modules $H^i_{\fm}(R)$, $i
\in \mathbb{N}$, and the injective hull of the residue field of a
local ring $\ringR$, denoted here by $E=E_R(k)$. Many applications to
problems either coming from tight closure theory in commutative
algebra or from positive characteristic algebraic geometry have been
found this way. Lyubeznik and Smith have been naturally led to the
study of rings of Frobenius operators in relation to the localization
problem in tight closure theory, and asked whether the Frobenius ring
of operators on $E$ is finitely generated over $R$ in~\cite{LS}.
Katzman has shown that, in general, this is not true
in~\cite{K}. Later, \`{A}lvarez Montaner, Boix and Zarzuela showed
that infinite generation is common even among nice classes of rings
in~\cite{ABZ}. Other important aspects of the generation of
$\sF{(E)}$, including introducing the twisted construction, were
studied by Katzman, Schwede, Singh and Zhang recently in~\cite{KSSZ}. 

The goal of our paper is to formulate a new invariant for rings of
prime characteristic that describes the abundance of Frobenius
operators naturally associated to the ring. This concept allows us to
measure systematically the generation of the ring of Frobenius
operators, finite or infinite. We will be mainly concerned with
Frobenius operators on the injective hull of the residue field of the
ring. An interesting byproduct of our work is that the phenomenon
investigated here, the Frobenius complexity, appears to only be
loosely connected to tight closure theory, although Frobenius
operators were first studied in relation to it. This seems to suggest
that the Frobenius complexity of a ring is a new feature of prime
characteristic rings in addition to those coming from tight closure
theory. 

In this paper, we prove that the ring of Frobenius operators of a
zero-dimensional ring is finitely generated, see Theorem~\ref{zero}. 
We answer positively a conjecture on $\mathbb{Q}$-Gorenstein
rings posed by Katzman, Schwede, Singh and Zhang in~\cite{KSSZ}, see
Theorem~\ref{thm:conj}. 
Under certain conditions on its anticanonical cover, we show that the
Frobenius complexity of a local complete normal ring is finite, see
Theorem~\ref{thm:fg-finite}. Then in Theorem~\ref{sally}, we show
that the complexity of the ring is less than or equal to zero for
rings of dimension at most two.   
Finally, in Theorem~\ref{thm:segre}, we show 
that even nice rings, such as 
determinantal rings, which are Cohen-Macaulay and strongly F-regular,
can have strictly positive complexity.

Let us review some of the main notation that will be used in this
paper. For any $e \geq 0$, we let $R^{(e)}$ be the $R$-algebra defined
as 
follows: as a ring $\R{e}$ equals $R$ while the $R$-algebra structure
is defined by $r \cdot s = r^{q} s$, for all $r \in R,\, s \in
R^{(e)}$. Note that when $R$ is reduced 
we have that $\R{e}$ is isomorphic to $R^{1/q}$ as $R$-algebras. Also,
$\R{e}$ as an $\R{e}$-algebra is simply $R$ as an
$R$-algebra. Similarly, for an $R$-module $M$, we can define 
a new $R$-module structure on $M$ by letting $r * m = r^{p^e}m$, for
all $r\in R,\,  m\in M$. We denote this
$R$-module by $\M{e}$. For example, given an ideal $I$ of $R$, we have
$ \R{e} \otimes_R R/I$ %fixed this 
is (naturally isomorphic to) $R/I^{[q]}$, in which $I^{[q]}$ is the
ideal of $R$ generated by $\{x^q : x \in I\}$. 

Consider now an $e$th Frobenius action on $M$, $\phi: M \to M$. This map can naturally be identified with an $R$-module homomorphism $\phi : M \to \M{e}$. It can be seen that such an action naturally defines an $R$-module homomorphism $f_{\phi}: \R{e} \otimes_R M \to M$, where $f_{\phi} (r \otimes m) = r \phi(m)$, for all $r \in R,\, m \in M$. Here, $\R{e}$ has the usual structure as an $R$-module given by $\R{e}=R$ on the left, while on the right we have the twisted Frobenius action.

%fixed the paragraph below

Let $\sF^e(M)$ be the collection of all $e$th Frobenius operators on $M$. We have a natural $R$-module isomorphism:
$$\sF^e{(M)} = \Hom{R}{M}{\M{e}} \simeq \Hom{R}{\R{e} \otimes_R M}{M},$$
defined by $P(\phi) = f_{\phi}$. The $R$-module structure on $\sF^e{(M)}$ is given by $(r\phi)(x) = r\phi(x)$ for every $r \in R,\, \phi  \in \sF^e{(M)}$ and $x \in M$.

Note that $P$ is additive and $P( r' \phi) (r \otimes m)= r((r'\phi)(m)) = r (r'\phi(m))= rr'\phi(m)= r'(r\phi(m)) = r'P(\phi)(r\otimes m)$. And so $P(r'\phi) = r'P(\phi)$, for all $r \in R$, all $\phi \in  \Hom{R}{M}{\M{e}}.$

\begin{Def}\label{def:sfe} %added this label

We define {\it the algebra of Frobenius operators} on $M$ by
$$\sF{(M)} = \oplus_{e\geq 0} \sF^e{(M)}.$$

The ring operation on $\sF{(M)}$ is given by composition of functions
(as multiplication). If $\phi \in \sF^e{(M)},\, \psi \in
\sF^{e'}{(M)}$ then $\phi \psi:=\phi\circ \psi \in \sF^{e+e'}(M)$. 
Note that $\phi \psi \neq \psi \phi$ in general.

The ring operation on $\sF{(M)}$ defines a module structure
$\sF^e{(M)}$ over $\sF^0{(M)} = {\rm End}_R(M)$. Since $R$ maps
canonically to $\sF^0{(M)}$, this makes $\sF^e{(M)}$ an $R$-module by
restriction of scalars. 
Note that $(\phi \circ r)(m) = \phi(rm) = (r^q\phi)(m)$, for
all $r \in R, m \in M$. Therefore, $\phi r = r^q \phi$, for all $r\in
R, \, \phi \in \sF^e{(M)}$, with $q=p^e$.
This indicates that, in general, ${\rm End}_R(M)$ is not in the center
of $\sF{(M)}$ and hence $\sF{(M)}$ is not an $R$-algebra using the
well established notion of an algebra. 
\end{Def}

\section{The Complexity of a Noncommutative Graded Ring}

\begin{Def}Let $A= \oplus_{e\geq 0} A_e$ be a $\mathbb{N}$-graded ring, not
  necessarily commutative. %, $A= \oplus_{e\geq 0} A_e$.
\begin{enumerate}
\item Let $G_e(A)= G_e$ be the subring of $A$ generated by the homogeneous elements of
  degree less than or equal to $e$. (So $G_0(A) = A_0$.) 
We agree that $G_{-1} = A_0$.

\item We use $k_e=k_e(A)$ to denote the minimal number of homogeneous
  generators of $G_e$ as a subring of $A$ over $A_0$. 
(So $k_0 = 0$.) We agree that $k_{-1} = 0$.
We say that $A$
  is {\it degree-wise finitely generated} if $k_e < \infty$ for all
  $e$. 

\item For a degree-wise finitely generated ring $A$, we say that a set
  $X$ of homogeneous elements of $A$ minimally generates $A$ if for all
  $e$, $X_{\leq e} =\{ a \in X: \deg(a) \leq e \}$ is a minimal set of
  generators for $G_e$ with $k_e = |X_{\leq e} |$ for every $e \ge 0$. Also,
  let $X_e= \{ a \in X: \deg(a)=e \}$. 
\end{enumerate}
\end{Def}

\begin{Rem} 
Let $A= \oplus_{e\geq 0} A_e$ be a $\mathbb{N}$-graded ring, not
necessarily commutative. 
\begin{enumerate}
%\item Clearly $k_{-1} = k_0 = 0$ and $X_{\le 0} = X_0 = \emptyset$. 
\item  Note that $G_e$ is $\mathbb{N}$-graded and $G_e \subseteq
  G_{e+1}$, for all $e \ge 0$. Also, $(G_e)_i= A_i$ for all $0
  \leq i \leq e$ and $(G_e)_{e+1} \subseteq A_{e+1}$. Moreover, both $A_i$
  and $(G_e)_i$ are naturally $A_0$-bimodules, for all $i,\, e$.  
\item Assume that $X$ minimally generates $A$. Then %it is clear that
  $|X_e | = k_{e} -k_{e-1}$ for all $e \ge 1$. 
\item Every degree-wise finitely generated $\mathbb N$-graded ring
  admits a minimal generating set; 
see Proposition~\ref{prop:minimal-gen} next. 
\end{enumerate}
\end{Rem}

\begin{Prop} \label{prop:minimal-gen}
With the notations introduced above,  
let $X$ be a set of homogeneous elements of $A$. Then
\begin{enumerate}
\item The set $X$ generates $A$ as a ring over $A_0$ if and only if 
$X_{\le e}$ generates $G_e$ as a ring over $A_0$ for all $e \ge 0$ 
if and only if the image of $X_{e}$ generates
$\frac{A_{e}}{(G_{e-1})_{e}}$ as an $A_0$-bimodule for all $e \ge 0$.  
\item Assume that $A$ is degree-wise finitely generated
  $\mathbb{N}$-graded ring and $X$ generates $A$ as a ring over $A_0$. 
The set $X$ minimally generates $A$ as a ring over $A_0$ if and
  only if $|X_{e}|$ is the minimal number of generators (out of all
 homogeneous generating sets) of $\frac{A_{e}}{(G_{e-1})_{e}}$ as an
  $A_0$-bimodule for all $e \ge 0$.  
\end{enumerate}
\end{Prop}

\begin{proof}
Both statements follow from consideration of degree. Here is a proof
with details. 

(1) Assume that $X$ generates $A$ as a ring over $A_0$. For any $e \ge
0$ and any $a \in A_e$, by considering degree, we see that $a$ can be
written as an expression involving elements in $A_0 \cup X_{\le e}$.
Consequently
\[
a \in A_0X_{e}A_0 + (G_{e-1})_e,
\]
in which $A_0X_{e}A_0$ stands for the $A_0$-bimodule generated by
$X_e$. Thus $A_e = A_0X_{e}A_0 + (G_{e-1})_e$, which verifies that
the image of $X_{e}$ generates $\frac{A_{e}}{(G_{e-1})_{e}}$ as an
$A_0$-bimodule, for all $e \ge 0$. 

Conversely, assume that the image of $X_{e}$ generates
$\frac{A_{e}}{(G_{e-1})_{e}}$ as an $A_0$-bimodule for all $e \ge 0$. 
It follows that, for any $e \ge 0$ and any $a \in A_e$,  we have
\[
a \in A_0X_{e}A_0 + (G_{e-1})_e,
\]
which implies that $a$ is generated by $X_{\le e}$ over
$A_0$. Therefore $A$, as a ring, is generated by $X$ over $A_0$. 

(2) Suppose that there exists $e \ge 0$ such that $|X_{\le e}| > k_e$. We
may further assume that this $e$ is minimal with the property. If $e =
0$, then $|X_{\le 0}| > k_0 = 0$ and hence $|X_{\le 0}|$ is greater
than the minimal number of generators (out of all
  homogeneous generating sets) of $\frac{A_{0}}{(G_{-1})_{0}} = 0$. 
So assume $e > 0$. Then $|X_{\le e-1}| = k_{e-1} < \infty$, hence 
\[
|X_e| = |X_{\le e}| - |X_{\le e-1}| > k_e - k_{e-1}.
\]
Choose a set $Y$ of homogeneous elements such that $Y$ (minimally)
generates $G_e$ with $|Y| = k_e$. It follows that $Y_{\le e-1}$ generates
$G_{e-1}$, so that $|Y_{\le e-1}| \ge k_{e-1}$. 
(In fact $|Y_{\le e-1}| = k_{e-1}$.) Thus
\[
|X_e|  > k_e - k_{e-1} \ge |Y| - |Y_{\le e-1}| 
= |Y_{\le e}| - |Y_{\le e-1}| = |Y_e|.
\]
Applying (1) to $G_e$ and $Y$, we see that the image of $Y_e$
generates $\frac{(G_e)_e}{(G_{e-1})_{e}} = \frac{A_e}{(G_{e-1})_{e}}$
as an $A_0$-bimodule. This shows that $|X_{e}|$ is not the minimal
number of generators of $\frac{A_{e}}{(G_{e-1})_{e}}$ as an
$A_0$-bimodule.

Conversely, suppose that, for some $e \ge 0$, $|X_{e}|$ is not 
(hence greater than) the minimal number of generators of
$\frac{A_{e}}{(G_{e-1})_{e}}$ as an $A_0$-bimodule. 
% We further assume that $e$ is minimal with the property.
The case $e = 0$ is trivial. So we assume $e > 0$. 
There exists $Y_e \subseteq A_e$ with $|Y_e| < |X_e|$ such that $Y_e$
consists of homogeneous elements and the image of $Y_e$ generates
$\frac{A_{e}}{(G_{e-1})_{e}}$ as an 
$A_0$-bimodule. By (1) applied to $G_e$, we see
that $X_{\le e-1} \cup Y_e$ generates $G_e$. Therefore, either $|X_{\le e}|
= \infty > k_e$ or $|X_{\le e}| > |X_{\le e-1}| + |Y_e| \ge k_e$, implying
that $X$ does not minimally generate $A$ as a ring over $A_0$.
\end{proof}

\begin{Cor} \label{cor:minimal-gen}
Let $A$ be a degree-wise finitely generated $\mathbb{N}$-graded ring
and $X$ a set of homogeneous elements of $A$. Then
\begin{enumerate}
\item The minimal number of generators of $\frac{A_{e}}{(G_{e-1})_{e}}$ as an
  $A_0$-bimodule is $k_e - k_{e-1}$ for all $e \ge 0$. 
\item If $X$ is generates $A$ as a ring over $A_0$ then 
$|X_e| \ge k_e - k_{e-1}$ for all $e \ge 0$. 
\end{enumerate}
\end{Cor}

\begin{proof}
Both (1) and (2) follow from Proposition~\ref{prop:minimal-gen}.
\end{proof}

From the above, we see that the sequence $\{k_e\}_e$ or % rather
$\{k_e-k_{e-1}\}_e$ describes how far away $A$ is from being finitely
generated over $A_0$. This leads us to define the \emph{complexity} of
$A$ as follows. Recall that, for sequences $\{a_e\}_e$ and
$\{a'_e\}_e$, we write $a_e = O(a'_e)$ precisely when there exists $b
\in \mathbb R$ and $k \in \mathbb N$ such that $|a_e| \le |ba'_e|$ for
all $e \ge k$.  

\begin{Def}
\label{degreewise}
Let $A$ be a degree-wise finitely generated ring.
The sequence $\{k_e\}_e$ is called the {\it growth} sequence for
$A$. The {\it complexity} sequence is given by $\{ c_{e}(A)=
k_{e}-k_{e-1} \}_{e\geq 0}$.   
The {\it complexity} of $A$ is
$$\inf \{ n \in \mathbb{R}_{> 0}: c_e(A)=k_{e} - k_{e-1} = O(n^e) \}$$ 
and it is denoted by $\cx(A)$. If there is no $n >0$ such that
$c_e(A)= O(n^e)$, then we say that $\cx(A)=\infty$. 
\end{Def}

\begin{Rem}
\label{locoh}
Let $A$ be as above.
\begin{enumerate}
%removed first part of remark.
\item Note that $\cx(A) >0$ implies $\cx(A) \geq 1$. (Indeed, if $0 <
  n < 1$, then $n^e \to 0$ as $e \to \infty$. Thus $c_e(A)=k_{e} -
  k_{e-1} = O(n^e)$ with $0 < n < 1$ implies that $c_e(A)$ is
  eventually zero.)
\item
It is obvious that $\cx(A)=0$ if and only if the sequence 
$\{c_{e}(A)\}_{e\geq 0}$ is eventually zero if and only if $A$ is
finitely generated as a ring over $A_0$. 
\item Similarly, $\cx(A)=1$ if the sequence 
$\{c_{e}(A)\}_{e\geq 0}$ is bounded by above, but not eventually zero.
\item When $R$ is $d$-dimensional, complete and $S_2$, Lyubeznik and
  Smith have showed that $\sF(H^d_{\fm}(R))$ is generated by one
  element over $\sF^0(H^d_{\fm}(R))=R$, namely the canonical Frobenius
  action $F$ on $H^d_{\fm}(R)$, see Example 3.7 in~\cite{LS}. This
  shows that $\cx(H^d_{\fm}(R)) = 0$ for $d$-dimensional $S_2$ local
  rings. 
\end{enumerate}
\end{Rem}

\begin{Def}\label{def:nearly-onto}
Let $A$ and $B$ be $\mathbb N$-graded rings and $h \colon A \to B$ be
a graded ring homomorphism. We say that $h$ is \emph{nearly onto} if 
$B = B_0[h(A)]$ (that is, $B$ as a ring is generated by $h(A)$ over $B_0$). 
\end{Def}

\begin{Thm}%[See Lemma~\ref{onto}]
\label{thm:nearly-onto}
Let $A$ and $B$ be $\mathbb N$-graded rings that are degree-wise
finitely generated. If there
exists a graded ring homomorphism $h \colon A \to B$ that
is nearly onto, then $c_e(A) \ge c_e(B)$ for all $e \ge 0$.
\end{Thm}

\begin{proof}
Choose a set $X$ of homogeneous elements of $A$ such that $X$
minimally generates $A$ as a ring over $A_0$. Since $h \colon A \to B$ 
is nearly onto, we see
\[
B = B_0[h(A)] = B_0[h(X)].
\]
This implies that $h(X)$ is a set of homogeneous elements of $B$ and
moreover $h(X)$ generates $B$ as a ring over $B_0$. 
By Corollary~\ref{cor:minimal-gen}, we see
\[
c_e(A) = |X_e| \ge |h(X_e)| = |(h(X))_e| \ge c_e(B)
\]
for all $e$. 
\end{proof}

\begin{Def}
Let $A$ be a $\mathbb{N}$-graded ring such that there exists a ring homomorphism $R \to A_0$, where $R$ is a commutative ring. We say that $A$ is a (left) $R$-{\it skew algebra} if $aR \subseteq Ra$ for all homogeneous elements $a \in A$. A right $R$-skew algebra can be defined analogously. In this paper, our $R$-skew algebras will be left $R$-skew algebras and therefore we will drop the adjective ``left'' when referring it to them.

\end{Def}

\begin{Cor}
\label{interpretation}
Let $A$ be a degree-wise finitely generated $R$-skew algebra such that
$R=A_0$. Then $c_{e}(A)$ equals the minimal number of generators of
$\frac{A_{e}}{(G_{e-1})_{e}}$ as a left $R$-module for all $e$.
\end{Cor}

\begin{proof}
This follows from Corollary~\ref{cor:minimal-gen} and the fact that
$A_0XA_0 = RXR= RX$ for any set $X$ of homogeneous elements of $A$.
\end{proof}

An important example for us is the ring of Frobenius operators on an $R$-module $M$. Note that there exists a canonical homomorphism  $R \to \sF^0{(M)} = {\rm End}_R(M)$. The main example is the case of a complete local ring $\ringR$ and $M=E=E_R(k)$, the injective hull of the residue field of $R$ as an $R$-module. In this case $\sF{(E)}$ is an $R$-skew algebra and $R= \sF^0{(E)}$.

Next, we prove that $\sF{(E)}$ is finitely generated over $R$ when
$\dim(R) = 0$. In fact, we will prove a more general result concerning
$\sF{(M)}$. 

\begin{Thm}\label{zero}
If $M$ is an $R$-module with finite length, then $\sF(M)$ is finitely
generated over $\sF^0(M) = \hom_R(M,M)$ (and also over $R$ via the
natural ring homomorphism $R \to \hom_R(M,M)$).
\end{Thm}

\begin{proof}
By replacing $R$ with $R/\ann(M)$, we may assume that $R$ is
Noetherian with $\dim(R) = 0$. Then, as every such ring is a direct
product of finitely many $0$-dimensional local rings, we may further
assume that $R = (R, \fm, k)$ is a Artinian local ring without loss of
generality. 

Fix a set of minimal generators $\{b_1, \dotsc, b_r\}$ for $M$ as an
$R$-module, so that the images of $b_i$ form a basis for $M/\fm M$.

There exists $e_0 \in \mathbb N$ such that $\fm^{[p^{e_0}]} \subseteq
\ann(M)$. (In fact $\fm^{[p^{e_0}]} = 0$, given that $M$ is actually
faithful after the consideration in the last paragraph.)
Thus the maximal ideal $\fm$ annihilates $\M{e}$ for $e \ge e_0$. 

Let $e$ be an arbitrary integer such that $e \ge e_0$. By the last
paragraph, every map in $\sF^e(M) = \hom_R(M,  \M{e})$ factors through
$M/\fm M$. Thus, for any (arbitrarily) chosen elements 
$m_j \in M$ with $j = 1, \dotsc, r$, there is a (unique) map in $\sF^e(M) =
\hom_R(M, \M{e})$ such that $b_j \mapsto m_j$ for all $j = 1,
\dotsc, r$. In particular, for any given $r \times r$ matrix $A =
(a_{ij} \in M_{r \times r}(R)$, there is a (unique) map in $\sF^e(M) =
\hom_R(M, \M{e})$ such that $b_j \mapsto \sum_{i=1}^ra_{ij}b_i$ for
all $j = 1, \dotsc, r$. (Here the expression $\sum_{i=1}^ra_{ij}b_i$
only involves the scalar multiplication of $M$, although it represents
an element in $\M{e}$.) We agree to use $\A{e}$ to denote the map
determined by $A$. To summarize, for every $A \in  M_{r \times r}(R)$
there is a corresponding map $\A{e} \in \sF^e(M)$, although
different matrices could determine the same map. 
In particular, we have the map $ \mI^{(e)} \in \sF^e(M)$ arising from the
identity matrix $\mI^{(e)} \in M_{r \times r}(R)$. 
(Note that this depends very much on the choice of $\{b_1, \dotsc,
b_r\}$. But we have already fixed $\{b_1, \dotsc, b_r\}$ earlier.)

Also, it is clear that every map $\phi \in \sF^e(M) = \hom_R(M,
\M{e})$ arises this way, i.e., $\phi = \A{e}$ for some $A \in
M_{r \times r}(R)$; in fact this statement does not rely on $e
\ge e_0$.   

Moreover, for $\phi = \A{e} \in \sF^e(M)$ and $\psi = \B{e'} \in
\sF^{e'}(M)$, it is routine to verify that $\phi \psi =
(AB^{[p^e]})^{e+e'}$, in which $B^{[p^e]}$ stands for the matrix
derived from $B$ by raising all entries of $B$ to the $p^e$-th power.
In short, $(\A{e})(\B{e'}) = (AB^{[p^e]})^{(e+e')}$.

Now we are ready to prove that $\sF(M)$ is finitely generated over 
$\sF^0(M) = \hom_R(M,M)$. Indeed, we claim that $\sF(M)$, as a ring, 
is generated by $\sF^0(M), \sF^1(M), \dotsc, \sF^{2e_0 - 1}(M)$. (This
would suffice, as each $\sF^i(M)$ is a finitely generated left
$R$-module via the natural ring homomorphism $R \to \hom_R(M,M)$. This
would also prove that $\sF(M)$ is finitely generated over $R$.)
Let $\phi \in \sF^e(M)$ with $e \ge 2e_0$, so that $\phi = \A{e}$ for
some $A \in M_{r \times r}(R)$. Since $e-e_0 \ge e_0$, we see 
\[
\phi = \A{e} = (\A{e_0}) (\mI^{(e-e_0)}).
\]
(Here $\A{e_0} \in \sF^{e_0}(M)$ and $\mI^{(e-e_0)} \in \sF^{e-e_0}(M)$
are regarded as Frobenius actions on $M$, as we set up above.) 
It suffices to verify that $ \mI^{(e-e_0)}$ can be generated by 
$\sF^{e_0}(M), \dotsc, \sF^{2e_0 - 1}(M)$. If $e - e_0 \le 2e_0 -1$,
this is clear. If $e - e_0 \ge 2e_0$, we write $e - e_0 = e_0k + c$
with $1 \le k \in \mathbb Z$ and $e_0 \le c \le 2e_0 -1$, which
implies
\[
 \mI ^{(e-e_0)}= (\mI ^{(e_0k)}) ( \mI^{(c)})
=\left[(\mI ^{(e_0)})^k\right]( \mI^{(c)}).
\]
This completes the proof.
\end{proof}

\begin{Cor}
If $R$ is a $0$-dimensional local ring, then $\sF(E)$ is finitely
generated over $R = \sF^0(E)$. 
\end{Cor}

\begin{Def}Let $\ringR$ be a local ring. We define the {\it Frobenius
complexity} of the ring $R$ by 
\[
\cx_F(R) = \log_p (\cx(\sF{(E)})).
\]
Also, denote $k_e(R) : = k_e (\sF{(E)})$, for all $e$, and call these
numbers the {\it Frobenius growth sequence} of $R$. Then $c_e= c_e(R):
= k_{e}(R)-k_{e-1}(R)$ defines the {\it Frobenius complexity sequence}
of $R$. If the Frobenius growth sequence of the ring $R$ is eventually
constant, then the Frobenius complexity of $R$ is said to be
$-\infty$. 
If $\cx(\sF(E)) = \infty$, the Frobenius complexity of $R$ is said to
be $\infty$. 

The reader should note that we will not regard $\ringR$ as a
degree-wise finitely generated ring, so the definition of $k_e, c_e$
for $R$ will not conflict with Definition~\ref{degreewise}. 
\end{Def}

\begin{Rem}
\begin{enumerate}
\item
The Frobenius operator algebra $\sF{(E)}$ is finitely generated
over $R$ if and only if the Frobenius complexity of the ring $R$
equals $-\infty$.   

\item
If the Frobenius complexity sequence is bounded but not eventually
zero then the Frobenius complexity of the ring is $0$. 

\item
The completion $R$ of a Stanley-Reisner ring has zero Frobenius
complexity, when $\sF{(E)}$ is not finitely generated, by results of
\`{A}lvarez Montaner, Boix and Zarzuela, namely 
\cite[Proposition 3.4 and 3.1.2]{ABZ}. 

\item
When $\ringR$ is $d$-dimensional and  Gorenstein, we have that $\E =
H^d_{\fm}(R)$, and Remark~\ref{locoh} (2) shows that $\cx_F(R) =
-\infty$. 
\item
If $\ringR$ is normal, $\mathbb{Q}$-Gorenstein and the order of the
canonical module is relatively prime to $p$, $\sF(E)$ is finitely
generated over $\sF^0{(E)}$ by \cite[Proposition~4.1]{KSSZ}. Hence
$\cx_F(R) = -\infty$ in this case. 
(See Theorem~\ref{thm:conj} for the converse.)
 \end{enumerate}
\end{Rem}

\section{T-construction}

Katzman, Schwede, Singh and Zhang have introduced an important $\mathbb{N}$-graded ring in their paper~\cite{KSSZ}, which is an example of an $R$-skew algebra. We will study the complexity of this skew-algebra in this section, and apply these results to the complexity of the ring $R$ in subsequent sections.

First let us review the definition of this ring. Let $\sR$ be an $\mathbb{N}$-graded commutative ring of prime characteristic $p$ with $\sR_0 =R$. 

\begin{Def}
Let  $T_e = \sR_{p^e-1}$ and $T(\sR) = \oplus _e T_e= \oplus_{ e\geq 0} \sR _{p^e-1}$. Define a ring structure on $T(\sR)$ by 
$$ a *b = ab^{p^e},$$
for all $a \in T_e,\, b\in T_{e'}$.
\end{Def}

This operation together with the natural addition inherited from $\sR$ defines a noncommutative $\mathbb{N}$-graded ring. Note that $T_0= R$ and if $a \in T_e, r \in R$, then
$ a *r = a r ^{p^e} = r^{p^e}  a = r^{p^e} * a$, for all $e \geq 0$, and hence $T(\sR)$ is a skew $R$-algebra.

We are now interested in computing the complexity of $T(\sR)$, when $\sR= R[x_1, \ldots, x_d]$ a polynomial ring in $d$ variables over $R$. Note that $T_e=\sR_{p^e-1}$ is an $R$-free module with basis given by monomials of total degree $p^e-1$.

We will use the notations introduced in the previous section. We see
that $G_{e-1} = G_{e-1}(T(\sR))$ is an $R$-free module with basis
consisting of monomials that can be expressed as products (under $*$,
the multiplication of $T(\sR)$) of monomials of degree  $p^i-1$ where
$i \leq e-1$. So the $R$-basis of $(G_{e-1})_e$ consists of these
monomials of total degree $p^e-1$. 

In conclusion the $R$-module $\frac{T_e}{(G_{e-1})_e}$ is $R$-free
with a basis given by monomials of degree $p^e-1$ which cannot be
written as products of monomials of degree $p^i-1$, with $i \leq
e-1$. We will refer to this basis as the \emph{monomial basis} of
$\frac{T_e}{(G_{e-1})_e}$.

We introduce the following notation $c_{d,e} : = \rank_R(\frac{T_e}{(G_{e-1})_e})$. By Corollary~\ref{interpretation}, we see that $c_{d,e}= c_e(T(\sR))$.

Let $m = x_1^{a_1} \cdots x_d^{a_d}$ be a monomial in $T_e$, that is, a monomial of degree $p^e-1$. This monomial $m$ belongs to $(G_{e-1})_e$ if and only if it can be decomposed as 
$m = m' * m''$, where $m' \in T_{e_1},\, m''\in T_{e_2}$ with $1 \leq e_i < e$ and $e_1 +e_2 =e$.

In other words, $m \in (G_{e-1})_e$ if and only if there is a decomposition $$m = (x_1^{a'_1} \cdots x_d^{a'_d}) * (x_1^{a''_1} \cdots x_d^{a''_d}) = \prod_{i=1}^d x_i ^{a'_i + p^{e_1} \cdot a''_i},$$
for some $1 \leq e_1 <e,\, 1\leq e_2<e,\, e_1+e_2=e,\,\sum a'_i
=p^{e_1}-1,\, \sum a''_i = p^{e_2}-1$. 

At this stage it is helpful to introduce several notations:
For an integer $a \in \mathbb N$, if $a= c_{n}p^{n} + \cdots + c_1p +
c_0$ with $0 \leq c_i \leq p-1$ for all $0 \le i \le n$,
then we use $a = \overline{c_{n} \cdots c_0}$ to denote the base $p$
expression of $a$. 
Also, we write $a \trunc_e$ to denote the remainder of $a$
when divided by $p^e$. Thus, if $a = \overline{c_{n} \cdots c_0}$
%with $n \ge e-1$ %WLOG
then $a \trunc_e = \overline{c_{e-1} \cdots c_0}$ 
and we refer to this number as the $e$th truncation of $a$.

Therefore, for $m = x_1^{a_1} \cdots x_d^{a_d} \in T_e$, 
there is a decomposition 
$$m = (x_1^{a'_1} \cdots x_d^{a'_d}) * (x_1^{a''_1} \cdots x_d^{a''_d}) = \prod_{i=1}^d x_i ^{a'_i + p^{e_1} \cdot a''_i},$$
for some $1 \leq e_1 <e,\, 1\leq e_2<e,\, e_1+e_2=e$, $\sum a'_i
=p^{e_1}-1,\, \sum a''_i = p^{e_2}-1$  
if and only if there exists an integer $1 \leq e_1 \leq e-1$  such that
$$ a_1 \trunc_{e_1} + \cdots + a_d \trunc_{e_1} = p^{e_1}-1.$$ 
It can readily be seen that this is further equivalent to
$$ a_1 \trunc_{e_1} + \cdots + a_{d-1} \trunc_{e_{1}} \leq p^{e_1}-1,$$
by dropping the part involving $a_d$. (For the very last equivalence,
the forward implication is trivial. For the backward direction, assume 
$a_1 \trunc_{e_1} + \cdots + a_{d-1} \trunc_{e_{1}} \leq p^{e_1}-1$, which
readily yields 
\[
a_1 \trunc_{e_1} + \cdots + a_{d} \trunc_{e_{1}} \le p^{e_1}-1+a_{d} \trunc_{e_{1}}
\le p^{e_1}-1+p^{e_1}-1. \tag{$\dagger$}
\]
Also note that $a_i \trunc_{e_1} \equiv a_i \mod p^{e_1}$ for all
$i = 1,\,\dotsc,\,d$. Thus the blanket assumption 
$a_1 + \cdots + a_{d} = p^{e}-1$ implies 
\[
a_1 \trunc_{e_1} + \cdots + a_{d} \trunc_{e_{1}} \equiv (a_1 + \cdots +
a_{d})\trunc_{e_{1}}= (p^{e}-1)\trunc_{e_{1}} \equiv p^{e_1}-1 \mod
p^{e_1}. \tag{$\ddagger$}
\] 
With $(\dagger)$ and $(\ddagger)$, the only possible choice for 
$a_1 \trunc_{e_1} + \cdots + a_{d} \trunc_{e_{1}}$ is $p^{e_1}-1$.)

In conclusion, we have the following 

\begin{Prop} \label{prop:bad}
A monomial $m = x_1^{a_1} \cdots x_d^{a_d}$ 
in $T_e$ gives an element of the monomial basis of 
$\frac{T_e}{(G_{e-1})_e}$ if and only if, for all $1 \leq e_1 <e$, 
\[
a_1 \trunc_{e_1} + \cdots + a_{d-1} \trunc_{e_{1}} + a_{d} \trunc_{e_{1}} \geq p^{e_1}
%\quad \text{and} \quad a_1 +\cdots + a_{d-1} + a_{d} = p^e-1
\]
if and only if, for all $1 \leq e_1 <e$,
$$ 
a_1 \trunc_{e_1} + \cdots + a_{d-1} \trunc_{e_{1}} \geq p^{e_1}.
%\quad \text{and} \quad a_1 +\cdots + a_{d-1} \leq p^e-1.
$$
\end{Prop}

Recall that for $a = \overline{c_{n} \cdots c_0}$, its $e_1$th
truncation is 
$$a \trunc_{e_1} = c_{e_1-1}p^{e_1-1}+ \cdots + c_0
= \overline{c_{{e_1}-1} \cdots c_0}.$$ 
This leads to the following reformulation of the result obtained above:

\begin{Prop}
\label{count}
A monomial $m =x_1^{a_1} \cdots x_d^{a_d}$ of $T_e$ gives an element
of the monomial basis of $\frac{T_e}{(G_{e-1})_e}$ if and only if, for
all $1 \leq e_1 <e$, the sum of the $e_1$th truncation of $a_i$,
$1\leq i \leq d$ carries over (to the digit corresponding to
$p^{e_1}$) in base $p$
% , while $a_1 + \cdots +a_{d-1}$, written in base
% $p$, does not carry over (to the digit corresponding to $p^{e}$)
if and only if, for
all $1 \leq e_1 <e$, the sum of the $e_1$th truncation of $a_i$,
$1\leq i \leq d-1$ carries over (to the digit corresponding to
$p^{e_1}$) in base $p$.
% , while $a_1 + \cdots +a_{d-1}$, written in base
% $p$, does not carry over (to the digit corresponding to $p^{e}$).  
\end{Prop}

\begin{Prop}\label{prop:3variables}
Let $p \geq 2$ be prime and $d=3$.
Then $$c_{3, e} 
= \sum_{0 \leq n_0, \ldots, n_{e-1} \leq p-1} n_0(n_1+1)\cdots (n_{e-2}+1)n_{e-1} 
= (1/2)^e p^e (p-1)^2(p+1)^{e-2}.$$
\end{Prop}

\begin{proof}
We are going to use Proposition~\ref{count} to find the number of
$x^iy^jz^k \in T_e = (T(R[x,y,z]))_e$ that form the monomial basis of
$\frac{T_e}{(G_{e-1})_e}$. Note that $d-1=2$, so for 
each $ 0 \leq i \leq p^e-1$ with $i=\overline{c_{e-1} \cdots c_0}$, 
we need to count the number of $j = \overline{a_{e-1} \cdots a_0}$ such that 
$$ c_0 + a_0 \geq p,$$
$$c_r + a_r \geq p-1 \text{ for } 1 \leq r \leq e-2,$$
$$c_{e-1} + a_{e-1} \leq p-2.$$
This number can be calculated to be 
$c_0(c_1+1) \cdots (c_{e-2}+1)(p-1-c_{e-1})$. %, denoted by $\xi_e(i)$.
So we have
\begin{align*}
c_{3, e} %= \sum_{0 \le i \le p^e -1} \xi_e(i)
&=\sum_{0 \leq c_0, \ldots, c_{e-1} \leq p-1}c_0(c_1+1) \cdots (c_{e-2}+1)(p-1-c_{e-1})\\
&=\sum_{0 \leq n_0, \ldots, n_{e-1} \leq
  p-1}n_0(n_1+1)\cdots(n_{e-2}+1) n_{e-1} \qquad \text{(by relabeling)}\\
&=\sum_{n_0 = 0}^{p-1}n_0 \cdot \sum_{n_1 = 1}^{p}n_1 
\dotsm \sum_{n_{e-2} =1}^{p}n_{e-2} \cdot \sum_{n_{e-1} =0}^{p-1}n_{e-1} \\
&=(p(p-1)/2)^2((p+1)p/2)^{e-2}\\
&=(1/2)^e p^e (p-1)^2(p+1)^{e-2}.\qedhere
\end{align*}
% After some routine calculations, we get
% \[c_{3,e} = (1/2)^e p^e (p-1)^2(p+1)^{e-2}.\qedhere
% \]
\end{proof}

\begin{Cor}\label{cor:3variables}
Let $p \geq 2$ be prime and $\sR = R[x_1, x_2, x_3]$. 
Then $\cx(T(\sR)) = \frac{p(p+1)}2 = \binom{p+1}{2}$.
\end{Cor}

\begin{proof}
Note that by definition % of $\cx(T(\sR))$ is equivalent to
$$\cx(T(\sR)=\inf \{ n >0 : c_{3,e} = O(n^e)\}.$$ 
A quick check verifies the claim.
\end{proof}

\begin{Prop}\label{more_variables}
We have the following formulas concerning
$c_{d,e}= c_e(T(R[x_1, \dotsc, x_d]))$:
\begin{enumerate}
\item When $e=1$, $c_{d,1} = \binom{d+p-2}{p-1}$.
\item
If $0 \leq d \leq 2$, then $c_{d,e}=0$ for all $e \ge 2$.
\item
If $d \geq 3$, then for $e \geq 1$
$$c_{d,e} \geq  \sum_{i=0}^{p^e-1}  \xi_{e}(i) \binom{d-3+i}{i},$$ where
$\xi_e(i) = (p-1-c_{e-1})(c_{e-2}+1)\cdots(c_2+1)(c_1+1)c_0$ if $i
=\overline{c_{e-1}\cdots c_0}$ in base $p$. 
\end{enumerate}
\end{Prop}

\begin{proof}
(1) By the definition of $c_{d,e}$, we see that $c_{d,1}$ is the number
of monomials of degree $p-1$ in $d$ variables, which is
$\binom{d+p-2}{p-1}$. 

(2) This follows from Proposition~\ref{prop:bad} or Proposition~\ref{count}.

(3) First, recall that in the proof of
Proposition~\ref{prop:3variables} where $d = 
3$, for each $e \ge 1$ and for each $i$ between $0$ and $p^e-1$, there are
$\xi_e(i)$ many monomials of the form $x^iy^jz^k$ that constitute part of
the monomial basis of $\frac{T_e}{(G_{e-1})_e}$ for $T(R[x,y,z])$. 

Second, note that if a monomial $x^iy^jz^k \in T(R[x,y,z])_e$ is in
the monomial basis for 
$\frac{T_e}{(G_{e-1})_e}$ for $T(R[x,y,z])$ then any monomial 
$x_1^{a_1}\dotsm x_{d-2}^{a_{d-2}}x_{d-1}^jx_d^k 
\in T(R[x_1,\dotsc, x_{d-2}, x_{d-1},x_d])_e$ with $a_1 + \dotsb +
a_{d-2} = i$ is in the monomial basis of $\frac{T_e}{(G_{e-1})_e}$ for
$T(R[x_1,\dotsc, x_{d-2}, x_{d-1},x_d])$. 
(To see this, study the contrapositive.)

Finally, observe that there are precisely $\binom{d-3+i}{i}$ many
strings $(a_1, \dotsc, a_{d-2}) \in \mathbb N^{d-2}$ such that 
$a_1 + \dotsb + a_{d-2} = i$. 

Combining the above, we conclude that there are at least 
$\sum_{i=0}^{p^e-1}  \xi_{e}(i) \binom{d-3+i}{i}$ many
monomials in $T(R[x_1,\dotsc, x_{d-2}, x_{d-1},x_d])_e$ that form a
portion of the monomial basis of $\frac{T_e}{(G_{e-1})_e}$ for
$T(R[x_1,\dotsc, x_{d-2}, x_{d-1},x_d])$. 
Hence the inequality.
\end{proof}

Next, we are going to give another way to formulate $c_{d,e}=
c_e(T(\sR))$ for $\sR = R[x_1, \dotsc, x_d]$. 

\begin{Prop}\label{prop:d}
For $\sR = R[x_1, \dotsc, x_d]$, we have the following formula for
$c_{d,e}= c_e(T(\sR))$: 
\[
c_{d,e}= %c_e(T(\sR)) = 
\sum_{\substack{(d_{e-1}=0,\, d_{e-2},\, \dotsc,\, d_0,\, d_{-1}=0) \in \mathbb N^{e+1}
\\d_n >  0 \text{ for } 0\le n < e-1}}
\prod_{n=0}^{e-1} M_d(d_np - d_{n-1} + p-1),\; \forall e \ge 1,
\]
in which $M_d(m)$ stands for the rank of $(R[x_1,
\dotsc, x_d]/(x_1^p, \dotsc, x_d^p))_m$ as an $R$-module. 
\end{Prop}

\begin{proof}
It suffices to count the number of monomials that produce
the monomial basis of $\frac{T_e}{(G_{e-1})_e}$, in which $T$ is short
for $T(\sR) = T(R[x_1, \dotsc, x_d])$. 
Let $1 \le e \in \mathbb N$ and 
$\ul x^{\ul a} := x_1^{a_1} \dotsm x_d^{a_d} \in T_e$, in which
$\ul a := (a_1, \dotsc, a_d) \in \mathbb N^d$ with $|\ul a| := a_1 +
\dotsb + a_d = p^e -1$. 
For each $i \in \{1,\dotsc,d\}$, write 
$a_i = \ol{a_{i,e-1} \dotsb a_{i,0}}$ 
in base $p$ expression. Then, for each $0 \le n \le e-1$, denote 
$\ul a_n := (a_{1,n}, \dotsc, a_{d, n})$, which can be referred to as
the vector of digits corresponding to $p^n$. Also, for each $0 \le n
\le e-1$, let $d_n(\ul a)$ denote the (accumulated) carry-over to the digit
corresponding to $p^{n+1}$ when computing $\sum_{i=1}^da_i$ in base $p$. 
In other words, $d_n(\ul a)$ is the carry-over to the digit
corresponding to $p^{n+1}$ when computing $a_1 \trunc_{n+1} + \dotsb +
a_d \trunc_{n+1}$. Denote $d(\ul a) :=(d_{e-1}(\ul a), \dotsc, d_{0}(\ul a))$.
% maybe an example would help the reader understand.
Note that $d_{e-1}(\ul a) = 0$ since $|\ul a| = a_1 +
\dotsb + a_d = p^e -1$.

Given $\ul x^{\ul a} \in T_e$ and $\ul d = (d_{e-1}, d_{e-2} , \dotsc,
d_0)$ with $d_{e-1} = 0$, 
it is straightforward to see that $d(\ul a) = \ul d$ if and only if 
\[
|\ul a_n| = d_np - d_{n-1} + p-1 \quad 
\text{for all} \quad n \in \{0, \dotsc, e-1\},
\]
in which we agree that $d_{-1} = 0$. 

By Proposition~\ref{count}, the image of $\ul x^{\ul a}$ is an element
of the monomial basis of $\frac{T_e}{(G_{e-1})_e}$ if and only if 
\[
d_n(\ul a) > 0 \quad \text{for} \quad 0 \le n < e-1, 
% \quad \text{ and } \quad d_{e-1} = 0,
\]
given that $|\ul a| = p^e -1$ and (hence) $d_{e-1}(\ul a) = 0$. 

Therefore, we can formulate $c_{d,e}=c_e(T(R[x_1, \dotsc, x_d]))$, 
$e \ge 1$, as follows 
(with the agreement $d_{-1} = 0 = d_{e-1}$ and $|\ul a| = p^e-1$):
\begin{align*}
c_{d,e} & 
= \sum_{\substack{(d_{e-1}=0,\, d_{e-2},\, \dotsc,\, d_0,\, d_{-1}=0) \in \mathbb N^{e+1}
\\d_n >  0 \text{ for } 0 \le n < e-1}} 
\left|\{\ul a \in \mathbb N^d : 
d(\ul a) = (d_{e-1}, d_{e-2} , \dotsc, d_0)
% \text{ and } |\ul a| = p^e-1
\}\right| \\
& = \sum_{\substack{(d_{e-1}=0,\, d_{e-2},\, \dotsc,\, d_0,\, d_{-1}=0) \in \mathbb N^{e+1}
\\d_n >  0 \text{ for } 0 \le n < e-1}} 
\left|\{\ul a \in \mathbb N^d : 
|\ul a_n| = d_np - d_{n-1} + p-1 \text{ for } 0 \le n \le e-1
%, \text{ and } |\ul a| = p^e-1
\}\right| \\
& = \sum_{\substack{(d_{e-1}=0,\, d_{e-2},\, \dotsc,\, d_0,\, d_{-1}=0) \in \mathbb N^{e+1}
\\d_n >  0 \text{ for } 0 \le n < e-1}}
\prod_{n=0}^{e-1} \left|\{\ul a_n \in [0,\,p-1]^d : 
|\ul a_n| = d_np - d_{n-1} + p-1\}\right|\\
& = \sum_{\substack{(d_{e-1}=0,\, d_{e-2},\, \dotsc,\, d_0,\, d_{-1}=0) \in \mathbb N^{e+1}
\\d_n >  0 \text{ for } 0 \le n < e-1}}
\prod_{n=0}^{e-1} M_d(d_np - d_{n-1} + p-1),
\end{align*}
in which $[0,p-1]:= \{0,\dotsc, p-1\}$ while, for any set $X$, $|X|$
stands for the cardinality of $X$.  
(Note that, in each of the summations above, almost all summands are zero.)
\end{proof}

Next, we outline a method that allows us compute $c_{d,e} =
c_e(T(R[x_1,\,\dotsc,\,x_d]))$ for any $d \ge 3$, in which $R$ may have
any prime characteristic $p$. (Note that, if $d \le 2$, then $c_{d,e}
= 0$ for all $e \ge 2$.)

\begin{Dis}\label{d,p}
Fix any $3 \le d \in \mathbb N$, any prime number $p$, and any ring $R$ with
characteristic $p$. Let $\sR=R[x_1,\,\dotsc,\,x_d]$.
We can determine $c_{d,e} = c_e(T(\sR))$ explicitly as follows:

Since $M_d(m)$ stands for the rank of 
$(R[x_1,\,\dotsc,\, x_d]/(x_1^p,\,\dotsc,\, x_d^p))_m$ over $R$, 
we see $M_d(m)=0$ whenever $m > d(p-1)$ or $m < 0$. 
In fact, all $M_d(m)$ can be read off from the following Poincar\'e
series (actually a polynomial):
\[
\sum_{m=0}^{\infty} M_d(m)t^m = \left(\frac{1-t^p}{1-t}\right)^d
=\left(1+\dotsb + t^{p-1}\right)^d.
\]

By Proposition~\ref{prop:d}, we have
\[
c_{d,e}=c_{e}(T(\sR))= 
\sum_{\substack{(d_{e-1}=0,\, d_{e-2},\, \dotsc,\, d_0,\, d_{-1}=0) \in \mathbb N^{e+1}
\\d_n >  0 \text{ for } 0 \le n < e-1}}
\prod_{n=0}^{e-1} M_d(pd_n - d_{n-1} + p-1),
\]
in which the nonzero summands come from
$(d_{e-1}=0,\, d_{e-2},\, \dotsc,\, d_0,\, d_{-1}=0) \in \mathbb N^{e+1}$
such that $d_n >0$ and  $0 \leq pd_n - d_{n-1} + p-1
\leq d(p-1)$ for all $0 \le n < e-1$. 

From $d_n > 0$ and $0 \leq pd_n - d_{n-1} + p-1 \leq d(p-1)$, 
$0 \le n < e-1$, we see 
$0< pd_n\leq (d-1)(p-1)+d_{n-1}$, which subsequently (and inductively)
implies that there is a uniform upper bound for all possible $d_n$. 
In fact, it is not hard to see that 
\[
1 \le d_n \le d-2 \quad \text{for all} \quad 0 \le n < e-1.
\]

For every $e \geq 0$, denote
\[
X_e =
\begin{pmatrix}
X_{e,1} \\
\vdots \\
X_{e,d-2}
\end{pmatrix},
\]
in which
\begin{align*}
X_{e,i} &=  
\sum_{\substack{( d_{e}=i,\, d_{e-1},\,\dotsc,\, d_0,\,d_{-1}=0) \in \mathbb N^{e+2}
\\d_n \in \{1,\,\dotsc,\,d-2\} \text{ for } 0 \le n \le e-1}}
\prod_{n=0}^{e} M_d(pd_n - d_{n-1} + p-1)
\end{align*}
for all $i \in \{1,\,\dotsc,\,d-2\}$.

With these notations, it is straightforward to see that, for all $1
\le i \le d-2$,  
\begin{align*}
X_{e+1,i} &= \sum_{j=1}^{d-2}M_d(pi-j+p-1)X_{e,j}
\end{align*}
In other words, $X_{e+1}$ can be computed recursively: 
\[
X_{e+1} = U \cdot X_e,
\quad \text{where} \quad 
U:= \begin{pmatrix}
u_{ij}
\end{pmatrix}_{(d-2) \times (d-2)} 
\quad \text{and} \quad 
u_{ij}=M_d(pi-j+p-1).
\]
Therefore, 
\[
X_e = U^{e}\cdot  X_0 \quad \text{for all} \quad e \ge 0.
\]

With $d$ and $p$ given, both $X_0$ and $U= (u_{ij})_{(d-2) \times
  (d-2)}$ can be determined explicitly. 
Accordingly, we can compute $X_e = U^{e}\cdot  X_0$ explicitly for
all $e \ge 0$. 

Finally, for all $e \ge 2$, we can determine  $c_{d,e} =c_{e}(T(\sR))$
explicitly, as follows: 
\begin{align*}
c_{d,e} =c_{e}(T(\sR)) 
&= \sum_{\substack{(d_{e-1}=0,\, d_{e-2},\, \dotsc,\, d_0,\, d_{-1}=0) \in \mathbb N^{e+1}
\\d_n \in \{1,\,\dotsc,\,d-2\} \text{ for } 0 \le n \le e-2}}
\prod_{n=0}^{e-1} M_d(pd_n - d_{n-1} + p-1) \\
&=\sum_{i=1}^{d-2} M_d(p \cdot 0 - i + p-1) X_{e-2,i}
= \sum_{i=1}^{d-2} M_d(p - i -1) X_{e-2,i}.
\end{align*}
Consequently, $\cx(T(\sR))$ can be computed.
\end{Dis}

To illustrate above method, we provide the following example, with  
$p=2$ and $d=4$.

\begin{Exa}
\label{2,4}
Let $\sR=R[x_1,x_2,x_3,x_4]$, with $R$ having characteristic $2$.
Since $M_4(m)$ stands for the rank of $(R[x_1,x_2, x_3, x_4]/(x_1^2,
x_2^2, x_3^2, x_4^2))_m$ as an $R$-module, we see $M_4(0)=1,\,
M_4(1)=4,\, M_4(2)=6,\, M_4(3)=4,\, M_4(4)=1$, and $M_4(m)=0$ for all
$m > 4$ or $m < 0$. 

Now, by Proposition~\ref{prop:d}, we have
\[
c_{4,e}=c_{e}(T(\sR))= 
\sum_{\substack{(d_{e-1}=0,\, d_{e-2},\, \dotsc,\, d_0,\, d_{-1}=0) \in \mathbb N^{e+1}
\\d_n >  0 \text{ for } 0 \le n < e-1}}
\prod_{n=0}^{e-1} M_4(2d_n - d_{n-1} + 1),
\]
in which the nonzero summands are the ones coming from
$(d_{e-1},\, d_{e-2},\, \dotsc,\, d_0,\, d_{-1}) \in \mathbb N^{e+1}$
such that $d_{-1}=d_{e-1}=0$, $d_n >0$, and  $0 \leq 2d_n - d_{n-1} + 1
\leq 4$ for all $0 \le n < e-1$. 

From $0 \leq 2d_n - d_{n-1} + 1 \leq 4$, $0 \le n < e-1$, we see 
$2d_n\leq 3+d_{n-1}$, which subsequently shows that 
\[
1 \le d_n \le 2 \qquad \text{for all} \quad 0 \le n < e-1.
\]
%(Actually, $d_0 = 1 = d_{e-2}$; but this is not needed in our proof.)

For every $e \geq 0$, let us denote
\begin{align*}
A_e &=  \sum_{\substack{( d_{e}=1,\, d_{e-1},\,\dotsc,\, d_0,\,d_{-1}=0) \in \mathbb N^{e+2}
\\d_n \in \{1,\,2\} \text{ for } 0 \le n \le e-1}}
\prod_{n=0}^{e} M_4(2d_n - d_{n-1} + 1),\\
B_e&= \sum_{\substack{( d_{e}=2,\, d_{e-1},\,\dotsc,\, d_0,\,d_{-1}=0) \in \mathbb N^{e+2}
\\d_n \in \{1,\,2\} \text{ for } 0 \le n \le e-1}}
\prod_{n=0}^eM_4(2d_n - d_{n-1} + 1).
\end{align*}
With these notations, we see that $A_e$ and $B_e$ can be computed recursively:
\begin{align*}
A_{e+1} &= M_4(2)A_e + M_4(1)B_e= 6A_e+4B_e,\\
B_{e+1} &= M_4(4)A_e+M_4(3)B_e= A_e + 4B_e.
\end{align*}
Let 
$X_e =
\begin{pmatrix}
A_e \\
B_e
\end{pmatrix}$
for all $e \ge 0$.
So we have 
$$
X_{e+1} = U \cdot X_e,
\quad \text{where} \quad 
U= \begin{pmatrix}
6 & 4 \\
1 & 4
\end{pmatrix}.
$$
By computing the eigenvalues and eigenvectors of $U$, we see that
$$ U = P \cdot D \cdot P^{-1},
% \qquad \text{where} \qquad
% D = \begin{pmatrix}
% 5+\sqrt{5} & 0\\
% 0 & 5-\sqrt{5} \end{pmatrix},  \quad
% P = \begin{pmatrix}
% 1+\sqrt{5} & 1-\sqrt{5}\\
% 1 & 1
% \end{pmatrix}.
$$
where $$D = \begin{pmatrix}
5+\sqrt{5} & 0\\
0 & 5-\sqrt{5} \end{pmatrix}  \quad \text{and} \quad
P = \begin{pmatrix}
1+\sqrt{5} & 1-\sqrt{5}\\
1 & 1
\end{pmatrix}.
$$
Therefore, 
\[
X_e = U^{e}\cdot  X_0= P\cdot D^{e} \cdot P^{-1} \cdot X_0 
= \frac{1}{2\sqrt{5}} \begin{pmatrix}
a_{11} & a_{12}\\
a_{21} & a_{22} 
\end{pmatrix} X_0,
\]
in which 
\begin{align*}
a_{11} &=(1+\sqrt{5})(5+\sqrt{5})^{e} - (1-\sqrt{5})(5 -\sqrt{5})^{e}, \\
a_{12} &= 4(5+\sqrt{5})^{e} -4 (5 -\sqrt{5})^{e}, \\
a_{21} &=(5+\sqrt{5})^{e} - (5 -\sqrt{5})^{e}, \\
a_{22} &= (\sqrt{5}+1)(5+\sqrt{5})^{e} + (1+\sqrt{5})(5 -\sqrt{5})^{e} .
\end{align*}
But $X_0 = \begin{pmatrix}
4\\
0
\end{pmatrix}.
$
Accordingly, we get 
\[
\begin{pmatrix}
A_e \\
B_e
\end{pmatrix}
=X_e = \frac{2}{\sqrt{5}} \begin{pmatrix}
(1+\sqrt{5})(5+\sqrt{5})^{e} - (1-\sqrt{5})(5 -\sqrt{5})^{e} \\
(5+\sqrt{5})^{e} - (5 -\sqrt{5})^{e}
\end{pmatrix},\; \forall e \ge 0.
\]
In conclusion, we see that, for all $e \ge 2$, 
\begin{align*}
c_{4,e} =c_{e}(T(\sR)) 
&= \sum_{\substack{(d_{e-1}=0,\, d_{e-2},\, \dotsc,\, d_0,\, d_{-1}=0) \in \mathbb N^{e+1}
\\d_n \in \{1,\,2\} \text{ for } 0 \le n \le e-2}}
\prod_{n=0}^{e-1} M_4(2d_n - d_{n-1} + 1) \\
&= M_4(0)A_{e-2} + M_4(-1)B_{e-2} = A_{e-2} \\
&= \frac{2}{\sqrt{5}}
\left((1+\sqrt{5})(5+\sqrt{5})^{e-2} - (1-\sqrt{5})(5 -\sqrt{5})^{e-2} \right).
\end{align*}
This shows that $\cx(T(\sR)) = 5 +\sqrt{5}$, concluding the example.
\end{Exa}

\section{Anticanonical cover}

Throughout this section, we let $\ringR$ be a local normal complete
ring. For a divisorial ideal $I$, i.e., an ideal of pure height one,
we denote $I\sy{n}$ its $n$th symbolic power. Let $\omega$ be a
canonical ideal of $R$. 

The anticanonical cover of the ring $R$ is defined as
$$ \sR = \sR(\omega) =  \oplus_{n \geq 0} \omega\sy{-n}.$$

The following theorem was recently proved by Katzman, Schwede, Singh
and Zhang in~\cite{KSSZ} and makes the transition from the
$T$-construction to ring of Frobenius operators on injective hull of
the residue field of $R$. 

\begin{Thm}[Katzman, Schwede, Singh, Zhang]
\label{anticanonical}
Let $\ringR$ as above, $E$ the $R$-injective hull of $k$ and $\omega$
its canonical ideal. Then there exists a graded isomorphism: 
$$ \sF(E) \simeq T(\sR(\omega)).$$
\end{Thm}

An easy consequence of the above theorem is the following result,
stated here for the convenience of the reader. 

\begin{Prop}
$\sF(E)$
is principally generated over $\sF^0(E)$ if and only if $\ord(\omega)
\mid p-1$. 
\end{Prop}

In \cite{KSSZ}, it is shown that if $R$ is $\mathbb Q$-Gorenstein with
$p \nmid \ord(\omega)$, then $\sF(E)$ is finitely generated over
$R$. It is conjectured in \cite{KSSZ} that if $p \mid \ord(\omega)$
then $\sF(E)$ is not finitely generated over $\sF^0(E)$. % $R$.

We give a confirmative proof to the conjecture of \cite{KSSZ}. 
We start with a lemma and a corollary.

\begin{Lma}\label{lma:frac_ideal}
Let $I_1,\,\dotsc,\,I_m$ and $J \cong R$ be fractional ideals of a local
domain $(R,\fm)$.  
\begin{enumerate}
\item If $I_1\dotsm I_m \cong R$ then $I_i \cong R$ for all $i =
  1,\,\dotsc,\,m$. 
\item Assume $I_1\dotsm I_m \subseteq J$. % and $J \cong R$. 
If $I_i \ncong R$ for some $i = 1,\,\dotsc,\,m$ then $I_1\dotsm I_m
\subseteq \fm J$
\end{enumerate}
\end{Lma}

\begin{proof}
(1) If $\prod_{i=1}^mI_i \cong R$ then all $I_i$ are projective
and hence free (of rank one) over $R$.

(2) If $I_i \ncong R$ for some $i = 1,\,\dotsc,\,m$, then
$\prod_{i=1}^mI_i \subsetneq J$ and hence 
$\prod_{i=1}^mI_i \subseteq \fm J$.
\end{proof}

We state the following corollary in such a way that it can be readily quoted
in the proof of Theorem~\ref{thm:conj}.

\begin{Cor}\label{cor:canonical}
Let $(R,\fm,k)$ be a $\mathbb Q$-Gorenstein normal domain with
$\ord(\omega) = m$. Then for all $r \in \mathbb Z$ and $0< s \in
\mathbb Z$ such that $m \nmid r$ and $m \mid rs$, we have
$(\omega^{(r)})^s \subseteq \fm \omega^{(rs)}$.
\end{Cor}
\begin{proof}
This follows from Lemma~\ref{lma:frac_ideal}, since $(\omega^{(r)})^s
\subseteq \omega^{(rs)} \cong R$ and $\omega^{(r)} \ncong R$. 
\end{proof}

\begin{Thm}\label{thm:conj}
Let $(R,\fm,k)$ be an excellent local normal domain with prime
characteristic $p$. 
If $R$ is $\mathbb Q$-Gorenstein with $\ord(\omega) = m$ such
that $p \mid m$, then $\sF(E)$ is not finitely generated over
$(\sF(E))_0 = \sF^0(E)$.
\end{Thm}

\begin{proof}
We may assume $R$ is complete. 
Since $\sF(E) \simeq T(\sR(\omega))$ by \cite{KSSZ}, it suffices to
prove that $T(\sR(\omega))$ is not finitely generated over
$(T(\sR(\omega)))_0 = R$, 
in which $\sR(\omega) =  \oplus_{n \geq 0} \omega\sy{-n}$ with the
(twisted) multiplication denoted by $*$. 
For shorter notation, denote
\[
T := T(\sR(\omega)) \quad \text{and} \quad 
T_e:= (T(\sR(\omega)))_e = \omega^{(1-p^e)}.
\]

Now it suffices to prove that $T_{\le e_0}$ does not generate $T$ for
any $e_0 \in \mathbb N$. 
To this end, fix an arbitrary $e_0\in \mathbb N$; and 
let $G: = G_{e_0}(T)$ denote the subring of $T$ generated by $T_{\le e_0}$.
For every $e > e_0$ satisfying $p^{e-e_0} > m = \ord(\omega)$, we have
the following:
\begin{align*}
G_e &\subseteq \sum_{i=0}^{e_0} (T_{e-i} * T_i) 
= \sum_{i=0}^{e_0} \left(T_{e-i}  T_i^{[p^{e-i}]}\right) \\
&= \sum_{i=0}^{e_0} \left(\omega^{(1-p^{e-i})}
  (\omega^{(1-p^{i})})^{[p^{e-i}]}\right)\\ 
&\subseteq \sum_{i=0}^{e_0} \left(\omega^{(1-p^{e-i})}
  (\omega^{(1-p^{i})})^{p^{e-i}}\right)\\ 
&= \sum_{i=0}^{e_0} \left(\omega^{(1-p^{e-i})}
  (\omega^{(1-p^{i})})^{p^{e-i}-m}(\omega^{(1-p^{i})})^m\right)\\ 
&\subseteq \sum_{i=0}^{e_0} \left(\omega^{(1-p^{e-i})}
  (\omega^{(1-p^{i})})^{p^{e-i}-m} \fm \omega^{((1-p^{i})m)}\right)
\quad \text{(by Corollary~\ref{cor:canonical})}\\ 
&= \fm \sum_{i=0}^{e_0} \left(\omega^{(1-p^{e-i})}
  (\omega^{(1-p^{i})})^{p^{e-i}-m} \omega^{((1-p^{i})m)}\right) 
\subseteq \fm \sum_{i=0}^{e_0} \omega^{(1-p^{e})}
%= \fm \omega^{(1-p^{e})} 
= \fm T_e \subsetneq T_e.
\end{align*}
Therefore $T_{\le e_0}$ does not generate $T$, which completes
the proof.
\end{proof}

As a corollary, we have the following result which deals with the total Cartier algebra of $R$. For a definition, we refer the reader to Definition 6.1 in~\cite{KSSZ}.

\begin{Thm}
Let $\ringR$ be an F-finite complete local normal
$\mathbb{Q}$-Gorenstein domain of prime 
characteristic $p$ with $p \mid \ord(\omega)$. 
Then the total Cartier algebra is not finitely generated over $R$.
\end{Thm}

\begin{proof}
This is because the total Cartier algebra is isomorphic to the opposite
of the Frobenius algebra as graded rings; see 2.2.1 in~\cite{ABZ}, for example. 
\end{proof}

We also state the following results on the complexity of a complete
local normal ring.

\begin{Thm}\label{thm:fg-finite}
Let $\ringR$ be a local normal complete ring and further
assume that $\sR = \sR(\omega)$ is a finitely generated $R$-algebra. 
Then $\cx_F(R) < \infty$.
\end{Thm}

\begin{proof}
Since $\sR$ is finitely generated over $R$, we can find a surjective
graded homomorphism: 
$$ \sS=R[x_1, \ldots, x_d] \to \sR,$$ where we let $\deg(x_i)=d_i$,
for all $1\leq i \leq d$. 

This induces a graded surjective homomorphism of skew algebras:
$T(\sS) \to T(\sR)$, and by Theorem~\ref{thm:nearly-onto} we get
$c_e(\sS) \geq c_e(\sR)= c_e(R)$. 
So we only need to show $\cx(T(\sS)) < \infty$.

But by construction $T(\sS) = \oplus _e \sS_{p^e-1}$.
According to our definitions 
$$c_{e}(T(\sS))  \leq \mu_R (\sS_{p^e-1}) \leq \binom{p^e-2+d}{d-1} 
= \frac{1}{(d-1)!} \cdot [p^e(p^e+1)\cdots (p^e+(d-2))].$$
This gives that
$$\cx(T(\sR)) \leq \cx(T(\sS)) \leq p^{d-1}.$$
So $\cx_F(R) \leq d-1$. 
\end{proof}

\begin{Rem}
The proof above indicates that, for any $\mathbb N$-graded commutative
ring $\sR$ of prime characteristic $p$, if $\sR$ is finitely generated
algebra over $\sR_0$ then $\cx(T(\sR)) < \infty$.
\end{Rem}

For rings of dimension at most two, we can be more precise due to a
result of Sally. 

\begin{Thm}[Sally,  Theorem~1.2, page~51 and Theorem~2.1, page 52 in \cite{Sa}]
\label{sally}
Let $\ringR$ be a local ring. Then
\begin{enumerate}
\item
The number of generators of all ideals is bounded above if and only if
$\dim(R) \le 1$;
\item
There is a bound on the number of generators of all ideals $I$ such
that $\fm$ is not an associated prime of $I$ if and only if $\dim(R) \le 2$.
\end{enumerate}
\end{Thm}

\begin{Thm}
Let $\ringR$ be a local, complete normal ring of dimension at most
$2$. Then $\cx_F(R) \leq 0$. 
\end{Thm}

\begin{proof} 
When $\dim(R) \le 1$, we know that $R$ is regular and hence $\cx_F(R)
= - \infty$ (see Remark~\ref{locoh}~(iii)). 
Nevertheless, the following argument works for $\Dim(R) \leq 2$.  

According to Theorem~\ref{anticanonical}, we have that $ \sF(E) \simeq
T(\sR(\omega))$, where $\sR(\omega)$ is the anticanonical cover of
$R$. The graded parts of the anticanonical cover are ideals of pure
dimension $1$. By Theorem~\ref{sally}, the number of generators of
these ideals are bounded by above. But Corollary~\ref{interpretation}
tells us now that $k_e(R)$ and hence $c_e(R)$ are bounded, and hence
$\cx_F(R) \leq 0$. 
\end{proof}

We close this section with the following relevant question.

\begin{Q}
Let $\ringR$ be a local complete normal ring. Is $\cx_F(R) < \infty$?
\end{Q}

\section{Examples of determinantal rings}

%All the rings have prime characteristic $p$.

Recall that, for $\mathbb N$-graded commutative rings 
$A = \oplus_{i \in \mathbb N}A_i$ and $B = \oplus_{i \in \mathbb N}B_i$
such that $A_0 = R = B_0$, their Segre product is 
\[
A \segre B = \oplus_{i \in \mathbb N}(A_i \otimes_R B_i),
\]
which is a ring under the natural operations.

In this section, we study the Segre product of $K[x_1, \dotsc, x_{d}]$
and $K[y_1,\dotsc, y_{d-1}]$. This ring is naturally isomorphic to the
determinantal ring $K[X]/I$ where 
$X$ is a $d \times (d-1)$ matrix of indeterminates and $I$ is the ideal
of the $2 \times 2$ minors of $X$.

\begin{Thm}[{\cite[(4.2.3), page 430]{Wa}}] \label{thm:Watanabe}
Let $K$ be a field and $d \ge 3$. 
The anticanonical cover of the Segre product of $K[x_1, \dotsc,
x_{d}]$ and $K[y_1, \dotsc, y_{d-1}]$ is isomorphic to 
\[
\bigoplus_{i \in \mathbb N} 
\left(\bigoplus_{\alpha \in \mathbb N^{d},\, 
\beta \in \mathbb N^{d-1},\, |\alpha| - |\beta|=i} 
Kx^{\alpha}y^{\beta}\right),
\]
in which the grading is governed by $i$. 
Here, for $\alpha = (a_1, \dotsc, a_d)$ and $\beta = (b_1, \dotsc,
b_{d-1})$ we denote $x^{\alpha} = x_1^{a_1} \dotsm x_d^{a_d}$ and
$y^{\beta} = y_1^{b_1} \dotsm y_{d-1}^{b_{d-1}}$.
\end{Thm}

\begin{Rem}\label{rmk:segre-complete}
Let $S_d$ denote the completion of 
$K[x_1,\dotsc, x_{d}] \segre K[y_1, \dotsc, y_{d-1}]$ 
with respect to the ideal generated by all homogeneous elements of
positive degree, in which $K$ is a field and $d \ge 3$.  
It is easy to see that
\begin{align*}
S_d &\cong \prod_{\alpha \in \mathbb N^{d},\, 
\beta \in \mathbb N^{d-1},\, |\alpha| = |\beta|} 
Kx^{\alpha}y^{\beta} \\
&=\left\{%\sum_{i=0}^{\infty}
\sum_{|\alpha| = |\beta|}a_{\alpha,\, \beta}x^{\alpha}y^{\beta} 
\; \Big | \; 
a_{\alpha,\, \beta} \in K, \, \alpha \in \mathbb N^{d},\, 
\beta \in \mathbb N^{d-1}\right\}
\subset K[[x_1, \dotsc,x_{d}, y_1, \dotsc, y_{d-1}]].
\end{align*}
Let $\mathscr R_d$ be the anticanonical cover of $S_d$.
It follows from Theorem~\ref{thm:Watanabe} that 
\[
\mathscr R_d \cong \bigoplus_{i \in \mathbb N} 
\left(\prod_{\alpha \in \mathbb N^{d},\, 
\beta \in \mathbb N^{d-1},\, |\alpha| - |\beta|=i} 
Kx^{\alpha}y^{\beta}\right),
\]
in which the grading is governed by $i$. 
\end{Rem}

\begin{Lma}\label{lma:nearly-onto}
Let $A$ and $B$ be degree-wise finitely generated $\mathbb N$-graded
commutative rings and   
$h \colon A \to B$ be a graded ring homomorphism.
\begin{enumerate}
\item The homomorphism $h$ is nearly onto if and only if $B_i$
is generated by $h(A_i)$ as a $B_0$-module for all $i \in \mathbb N$
(that is, $B$ is generated by $h(A)$ as a $B_0$-module).
\item If $A$ and $B$ have prime characteristic $p$ and $h$
is nearly onto, then the induced graded homomorphism  
$T(h) \colon T(A) \to T(B)$ is nearly onto.
\end{enumerate}
\end{Lma}

\begin{proof}
(1) It is clear that if $B$ is generated by $h(A)$ as a $B_0$-bimodule
then $h$ is nearly onto, which does not rely on $A$ or $B$ being
commutative. 

Now assume $h$ is nearly onto. 
%Let $b \in B_i$ with $i \in \mathbb N$. 
Since $B$ is commutative, it is routine to see that $B_i$
is generated by $h(A_i)$ as a $B_0$-module for every $i \in \mathbb N$. 

(2) If $h$ is nearly onto, then $B_i$ is generated by $h(A_i)$ as a
$B_0$-module for every $i \in \mathbb N$. In particular, $B_{p^e-1}$
is generated by $h(A_{p^e-1})$ as a $B_0$-module for every $e \in
\mathbb N$. Viewing this inside $T(B)$, we see that $T(B)_e$ is
generated by $T(h)(T(A)_e)$ as a left $T(B)_0$-module for every $e
\in \mathbb N$. Therefore $T(h)$ is nearly onto (by part (1) above). 
\end{proof}

\begin{Cor}\label{cor:nearly-onto}
Let $A$ and $B$ be $\mathbb N$-graded commutative rings of prime
characteristic $p$. 
If there exists a graded ring homomorphism $h \colon A \to B$ that
is nearly onto, then $c_e(T(A)) \ge c_e(T(B))$ for all $e \ge 0$.
\end{Cor}

\begin{proof}
This follows from Lemma~\ref{lma:nearly-onto} and
Theorem~\ref{thm:nearly-onto}.  
\end{proof}

\begin{Prop} \label{prop:nearly-onto}
Let $K$, $S_d$ and $\sR_d$ be as in Remark~\ref{rmk:segre-complete}
with $d \ge 3$. 
Then there are nearly onto graded ring homomorphisms from $\sR_d$ to
$K[x_1, \dotsc, x_{d}]$ and vice versa.
\end{Prop}

\begin{proof}
In light of Remark~\ref{rmk:segre-complete}, we simply assume 
\begin{align*}
\mathscr R_d = \bigoplus_{i \in \mathbb N} 
\left(\prod_{\alpha \in \mathbb N^{d},\, 
\beta \in \mathbb N^{d-1},\, |\alpha| - |\beta|=i} 
Kx^{\alpha}y^{\beta}\right).
\end{align*}
Define $\phi \colon \sR_d \to K[x_1, \dotsc, x_{d}]$ 
and $\psi \colon K[x_1, \dotsc, x_{d}] \to \sR_d$ by 
\begin{align*}
\phi(f(x_1,\dotsc, x_d,\, y_1, \dotsc, y_{d-1}))
&= f(x_1,\dotsc, x_d,\, 0, \dotsc, 0) \in K[x_1, \dotsc, x_{d}]\\
\text{and} \qquad 
\psi(g(x_1,\dotsc, x_d)) &= g(x_1,\dotsc, x_d) \in \sR_d,
\end{align*}
for all $f(x_1,\dotsc, x_d,\, y_1, \dotsc, y_{d-1}) \in \sR_d$ 
and all $g(x_1,\dotsc, x_d) \in K[x_1, \dotsc, x_{d}]$.

It is routine to verify that both $\phi$ and $\psi$ are graded ring
homomorphisms. 
As $\phi \circ \psi$ is the identity map, we see that $\phi$ is onto
and hence nearly onto.
Finally, note that for every $i \in \mathbb N$, $(\sR_d)_i$ is
generated by $\psi((K[ x_1, \dotsc, x_{d}])_i)$ as a module over
$(\sR_d)_0 = S_d$. So $\psi$ is nearly onto. 
This completes the proof.
\end{proof}

\begin{Thm}\label{thm:segre}
Let $K$, $S_d$ and $\sR_d$ be as in %Proposition~\ref{prop:nearly-onto} 
Remark~\ref{rmk:segre-complete}
with $d \ge 3$. Assume that $K$ has prime characteristic $p$. 
Then 
\begin{enumerate}
\item $T(\sR_d)$ and $T(K[x_1,\dotsc, x_{d}])$ have the same
complexity sequence. 
\item $\cx_F(S_d) = \log_p \cx(T(K[x_1,\dotsc, x_{d}]))$. 
\item $\sF(E_d)$ is not finitely generated over $\sF_0(E_d) \cong
  S_d$, where $E_d$ stands for the injective hull of the residue field
  of $S_d$ (over $S_d$).
\item The Frobenius complexity of $S_3$ is
$\cx_F(S_3) = 1 + \log_p(p+1) - \log_p 2$. Moreover, $\lim_{p \to \infty} \cx_F(S_3)=2$.
\item If $p = 2$, then $\cx_F(S_4) = \log_2(5 + \sqrt 5)$.
\end{enumerate}
\end{Thm}

\begin{proof}
Statements (1) and (2) follow from Corollary~\ref{cor:nearly-onto} and 
Proposition~\ref{prop:nearly-onto} in light of Theorem~\ref{anticanonical}. 
Statement~(3) follows from Proposition~\ref{more_variables} or
\cite{KSSZ}. Finally, (4) follows from Corollary~\ref{cor:3variables} 
while (5) follows from Example~\ref{2,4}.
\end{proof} 

\begin{Rem}\begin{enumerate}
\item
The statement in Theorem~\ref{thm:segre} (3) for $d=3$ has been proved
first in \cite[Section~5]{KSSZ}. It should be contrasted with (4) in
the above Theorem~\ref{thm:segre}. 

\item
This computation shows that even rings that are nice from the point of
view of tight closure theory, such as completions of determinantal
rings, can have positive Frobenius complexity. This means that, as $e
\to \infty$, there are more and more $e$th Frobenius actions on $\E$
that are fundamentally new (i.e., do not come from Frobenius actions
from lower degree). This phenomenon was illustrated in~\cite{KSSZ} as
well, but the Frobenius complexity provides a way to quantify
this. Interestingly, this invariant can be  
irrational and depends upon the characteristic $p$. 

\end{enumerate}
\end{Rem}


\begin{thebibliography}{GSTT}

\bibitem[ABZ]{ABZ} J.\ \`{A}lvarez Montaner, A.\ F.~Boix, S.\ Zarzuela, 
\emph{Frobenius and Cartier algebras of Stanley-Reisner rings}, 
J.~Algebra \textbf{358} (2012), 162--177.
\bibitem[Ka]{K} M.~Katzman 
\emph{An example of non-finitely generated algebra of Frobenius maps}, 
Proc.\ Amer.\ Math.\ Soc. \textbf{138} (2010), 2381-2383.
\bibitem[KSSZ]{KSSZ} M.\ Katzman, K.~Schwede, A.~K.~Singh, W.~Zhang, 
\emph{Rings of Frobenius operators}, Math.\ Proc.\ Cambridge Phil.\ Soc.\ 
\textbf{157} (2014), no.~1, 151--167.
\bibitem[LS]{LS} G.~Lyubeznik, K.~E.~Smith, 
\emph{On the commutation of the test ideal with localization and
  completion}, 
Trans.\ Amer.\ Math.\ Soc.\ \textbf{353} (2001), 3148--3180.
\bibitem[Sa]{Sa} J.~Sally, 
\emph{On the number of generators of ideals in local rings}, 
Marcel Dekker, New York and Basel, 1978.
\bibitem[Wa]{Wa} Kei-ichi Watanabe, 
\emph{Infinite cyclic covers of strongly F-regular rings}, 
Contemporary Mathematics \textbf{159} (1994), 423--432.  

\end{thebibliography}
\end{document}